\documentclass{amsart}
\usepackage{amsfonts,amssymb,amscd,amsmath,enumerate,verbatim,calc}
\usepackage[all]{xy}

\newcommand{\CM}{Cohen-Macaulay}

\newcommand{\D}{\mathbf{D}_{\bullet} }

\newcommand{\I}{\mathbf{I^\bullet} }
\newcommand{\K}{\mathbb{K}_{\bullet} }
\newcommand{\Kc}{\mathbb{K}^{\bullet} }

\newcommand{\C}{\mathbf{C} }

\newcommand{\X}{\mathbf{X} }

\newcommand{\F}{\mathbf{F}_{\bullet} }

\newcommand{\rt}{\rightarrow}

\newcommand{\bx}{\mathbf{x}}

\newcommand{\by}{\mathbf{y}}

\newcommand{\reg}{\operatorname{reg}}
\newcommand{\grade}{\operatorname{grade}}

\newcommand{\Tot}{\operatorname{Tot}}

\newcommand{\mindeg}{\operatorname{mindeg}}

\newcommand{\codim}{\operatorname{codim}}

\newcommand{\projdim}{\operatorname{projdim}}

\newcommand{\Hom}{\operatorname{Hom}}
\newcommand{\Ext}{\operatorname{Ext}}
\newcommand{\Tor}{\operatorname{Tor}}

\theoremstyle{plain}

\newtheorem{theorem}{Theorem}[section]
\newtheorem{corollary}[theorem]{Corollary}
\newtheorem{lemma}[theorem]{Lemma}
\newtheorem{proposition}[theorem]{Proposition}

\theoremstyle{definition}

\newtheorem{s}[theorem]{}
\newtheorem{remark}[theorem]{Remark}

\theoremstyle{remark}

\begin{document}
\title{Regularity of Koszul modules}
 \author{Tony~J.~Puthenpurakal }
\address{Department of Mathematics, Indian Institute of Technology Bombay, Powai, Mumbai 400 076, India}

\email{tputhen@math.iitb.ac.in}

\begin{abstract}
Let $K$ be a field and let $S = K[X_1, \ldots, X_n]$. Let $I$ be a graded ideal in $S$ and let $M$ be a finitely generated graded $S$-module.
We give  upper bounds on the regularity of Koszul homology modules $H_i(I, M)$ for several classes of $I$ and $M$.
\end{abstract}
\maketitle
\section{Introduction}
The Koszul complex is a fundamental construction in commutative algebra.
W. Vasconcelos writes in his book ``Integral closures'' \cite[page 280]{VasBook-05}:
\begin{quote}
"While the vanishing of the homology of a Koszul complex $K(\bx, M)$ is easy to track, the module theoretic properties
of its homology, with the exception of the ends, is difficult to fathom. For instance, just trying to see
whether a prime is associated to some  $H_i(\bx, M)$  can be very hard."
\end{quote}
The purpose of this paper is to enhance our knowledge about Koszul homology modules by establishing the following  results:

\s \label{setup} \emph{Setup:} Let $K$ be a field and let $S = K[X_1, \ldots, X_n]$. We give $S$ the standard grading. Let $I$ be a homogeneous ideal in $S$.
Let $\bx = x_1,\ldots,x_l$   a  system of homogeneous, not-necessarily minimal, generators of $I$ and let $\K(\bx)$ be the Koszul complex with respect  to $\bx$.  Assume
$|x_1| \geq |x_2| \geq \ldots \geq |x_l|$. Here $|x|$ denotes degree of $x$.  Let $g = \grade I$.  If $N$ is a finitely generated graded $S$-module then $\reg N$ denotes the regularity of $N$. Furthermore set
$t_i(N) = \reg \Tor^S_i(N,K)$ for $i \geq 0$.

We first prove
\begin{theorem}\label{m1}Let  $M$ be a  perfect $S$-module of grade $n-g$. Also assume that $\ell(M/IM)$ is finite. Then
$$\reg H_k(\bx, M) \leq \sum_{i=1}^{g+k}|x_i|   + t_{n-g}(M) - n.$$
\end{theorem}

Recall an ideal $I$ is said to be strongly Cohen-Macaulay if $H_i(I)$ is \CM \ for all $i$. For instance if $I$ is licci then it is strongly \CM, see \cite[1.14]{H}. As an easy application of Theorem \ref{m1} we get,
\begin{corollary}\label{scm}
If $I$ is strongly \CM \ ideal of grade $g$ then
$$ \reg H_k(\bx) \leq \sum_{i=1}^{g+k}|x_i|   - g. $$
\end{corollary}
We note that if $I$ is zero-dimensional then it is trivially strongly \CM.

Another application of Theorem \ref{m1} is
\begin{corollary}\label{cm}
If $I$ is a \CM \ ideal of grade $g$ then
$$ \reg S/I \leq \sum_{i=1}^{g}|x_i|   - g. $$
\end{corollary}
Next we show:
\begin{theorem}\label{m2}
Let $I$ be a zero-dimensional graded ideal in $S$ and let $M$ be a finitely generated graded $S$-module. Then for $k \leq \codim M$,
$$\reg H_k(\bx, M) \leq \sum_{i=1}^{n}|x_i| + t_k(M)  - n    $$
\end{theorem}

The main technique to prove our results is an analysis of a spectral sequence, discussed in section 3. Here is an overview of the contents of this paper. In section two we discuss a few preliminaries that we need. In section three we discuss the spectral sequence we need. In section four we give a proof of Theorem \ref{m1} and discuss its corollaries. In section five we prove Theorem \ref{m2}.

\section{Preliminaries}
In this section we discuss a few preliminary results that we need. In this paper all rings considered are Noetherian and all modules considered (unless otherwise stated) is finitely generated.

\s\label{reg-def} Let $K$ be a field and let $S = K[X_1,\ldots, X_n]$. We give the standard grading on $S$. Let $S_+$ be the irrelevant maximal ideal of $S$. Let $M$ be a graded $S$-module. Let $H^i_{S_+}(M)$ be the $i^{th}$-local cohomology of $M$ with respect to $S_+$. Then the regularity of $M$ is defined as
\[
\reg M =  \max\{ i + j \mid H^i_{S_+}(M)_j \neq 0 \}.
\]
We let $t_i(M) = \reg \Tor^S_i(M, K)$.
Let $\F$ be a graded minimal free resolution of $M$. Let ${\F}_i = \bigoplus_j S(-j)^{b_{ij}}$. Then
\[
\reg M = \max \{ j - i \mid b_{ij} \neq 0 \}.
\]

The following result is well-known and easy to prove.
\begin{proposition}\label{mod-reg}
(with hypotheses as in \ref{reg-def}) Let $u_1, \ldots u_r$ be linear forms in $S$ which is an $M$-regular sequence.
Then $\reg M/(u_1, \ldots, u_r)M = \reg M$.
\end{proposition}

\s We will need the vanishing of an element in the total complex of a bigraded complex.

\begin{lemma} \label{ehu}(\cite[2.6]{EHU}) Let $\X$ be any bounded (cohomological) double complex.
Suppose that $p, q$ are chosen so that the vertical homology of $\X$ is zero at
$X^{p',q'}$,
when $p' + q' = p + q$ and $q' < q$, and the horizontal homology of $\X$ is zero at
$X^{p',q'}$,
when $p' + q' = p + q$  and $q' > q$. If $\zeta \in  H^{p+q}(\Tot(\X))$ is represented by a cycle
$$z = (z^{p', q'}) \in \bigoplus_{p' + q' = p + q}X^{p', q'}   $$
satisfies $z^{p,q} = 0$, then $\zeta = 0$.
\end{lemma}
\begin{remark}
In \cite{EHU} note that in  a double complex $\X = \bigoplus_{p,q} X^{pq}$ the $p$ is the $y$-coordinate and $q$ is the $x$-coordinate. We hold the standard convention i.e., in a bicomplex $\X = \bigoplus_{p,q} X^{pq}$ the $x$-coordinate is $p$ and $y$-coordinate is $q$.
\end{remark}
\section{A spectral sequence}
Let $S= K[X_1,\ldots,X_n]$.
Let $\D \colon \cdots D_1 \rt D_0 \rt 0$ be a complex of finitely generated free $S$-modules.
Let $\I \colon  0 \rt I^0 \rt I^1 \rt \cdots$ be an injective resolution of $S$.
We consider the Hom co-chain complex $\C = \Hom(\D, \I)$; see \cite[2.7.4]{Weibel}.
Set $\D^* = \Hom(\D, S)$.
\begin{lemma}\label{first}
The first spectral sequence on $\C$ collapses and hence $H^i(\Tot(\C)) = H^i(\D^*)$
\end{lemma}
\begin{proof}
${}^{I}E^{pq}_0 = \Hom_S(D_p, I^q)$. So we get
\begin{align*}
{}^{I}E^{pq}_1 &= H^q(\Hom(D_p, \I)) \\
               &= \Hom(D_p, H^q(\I)) \\
               &= \begin{cases}0 & \text{for} \ q > 0; \\ \Hom(D_p, S) &\text{for} \ q = 0.   \end{cases}
\end{align*}
Therefore we get
\begin{align*}
{}^{I}E^{pq}_2 &= \begin{cases}0 & \text{for} \ q > 0; \\ H^p(\Hom(\D, S)) &\text{for} \ q = 0.   \end{cases} \\
                &= \begin{cases}0 & \text{for} \ q > 0; \\ H^p(\D^*) &\text{for} \ q = 0.   \end{cases}
 \end{align*}
\end{proof}
We now compute the $E_2$ term of the second standard spectral sequence on $\C$.
\begin{lemma}\label{second}
${}^{II}E^{pq}_2 = \Ext^p_S(H_q(\D), S)$
\end{lemma}
\begin{proof}
${}^{II}E^{pq}_0 = \Hom(D_q, I^p) $. So
\begin{align*}
{}^{I}E^{pq}_1 &= H^q(\Hom(\D, I^p)) \\
               &= \Hom(H_q(\D), I^p)
\end{align*}
The second equality is since $I^p$ is injective. Thus we get
\begin{align*}
{}^{II}E^{pq}_2 &= H^p \Hom(H_q(\D), \I) \\
                &= \Ext^p_S(H_q(\D), S).
\end{align*}
\end{proof}

The complex $\D$ we will consider is the following:\\
Let $\bx = x_1,\ldots,x_l$   a  system of homogeneous, not-necessarily minimal, generators of $I$ and let $\K(\bx)$ be the Koszul complex with respect  to $\bx$.  Assume
$|x_1| \geq |x_2| \geq \ldots \geq |x_l|$. Here $|x|$ denotes degree of $x$.
Let $\F \colon \cdots F_1 \rt F_0 \rt 0$ be a minimal resolution of $M$. Set $\D = \Tot(\K(\bx) \otimes \F)$.  Notice $\D^* = \Tot(\Kc\otimes \F^*)$.

\begin{lemma}\label{com-D}
$H_i(\D) = H_i(\bx, M)$
\end{lemma}
\begin{proof}
We use the first standard spectral sequence on $\K(\bx) \otimes \F$ to compute the homology of $\D$.
Set
$E^0_{pq} = K_p \otimes F_q$. So
\begin{align*}
E^1_{pq} &= H_q(K_p \otimes \F) \\
&= \Tor^S_q(K_p, M) \\
          &= \begin{cases} 0 & \text{for} \ q > 0; \\ K_p\otimes M & \text{for} \ q = 0.    \end{cases}
\end{align*}
Therefore we get
\[
E^2_{pq} = \begin{cases} 0 & \text{for} \ q > 0; \\  H_p(\bx,M)  & \text{for} \ q = 0.    \end{cases}
\]
Thus $H_i(\D) = H_i(\bx, M)$.
\end{proof}

\section{Proof of Theorem \ref{m1}}
In this section we prove Theorem \ref{m2}. We restate it for the convenience of the reader.
We assume that $I$ is a homogeneous ideal in $S$ of grade $g$. The module $M$ is perfect of grade $n-g$. We also assume that
$\ell(M/IM)$ is finite. Here $\ell(N)$ is the length  of the $S$-module $N$.
Notice that $\ell( H_i(\bx, M))$ is also finite for all $i \geq 0$.
Recall we are assuming $|x_1| \geq |x_2| \geq \cdots \geq |x_l|$.
\begin{theorem}\label{mt-1}
$\reg H_k(\bx, M) \leq \sum_{i=1}^{g+k}|x_i|   + t_{n-g}(M) - n$.
\end{theorem}
\begin{proof}
We consider the second spectral sequence  on $\C$. By \ref{second}  and \ref{com-D} we have
 ${}^{II}E^{pq}_2 = \Ext^p_S(H_q(\bx, M), S)$
Since $\ell( H_i(\bx, M))$ is also finite for all $i \geq 0$ we have
that ${}^{II}E^{pq}_2 = 0$ for all $p \neq n$. So the second spectral sequence on $\C$ collapses. By
\ref{second} it follows that
\begin{equation*}
\Ext^n_S(H_k(\bx, M), S) = H^{n+k}(\D^*) \quad \text{for all}\ k\geq 0. \tag{$\dagger$}
\end{equation*}
By local duality we have
\[
H_k(\bx, M) = \Hom_K\left(\Ext^n_S(H_k(\bx, M), S),  K \right)(n)
\]
So by ($\dagger$) we get that
\[
\reg H_k(\bx, M) = - \mindeg H^{n+k}(\D^*)  -n
\]
Notice $\D^* = \Tot(\K^* \otimes \F^*)$.
We prove that any homogenous element $\xi$ in $H^{n+k}(\D^*)$ of
\[
\deg \xi < - \sum_{i=1}^{g+k}|x_i|   - t_{n-g}(M)
\]
is zero.

Let $\xi = [z]$ where $z = (z^{p,q} \mid p+q = n+k)$ is a cycle.

Claim 1: $z^{g+k, n-g} = 0$.
\begin{align*}
\mindeg K^*_{g + k} \otimes F^*_{n-g} &=  \mindeg K^*_{g+k}  + \mindeg F^*_{n-g}  \\
                    &=   - \sum_{i=1}^{g+k}|x_i|  - t_{n-g}(M)  \\
                    &> \deg \xi
\end{align*}
So $z^{ g+k, n-g} = 0$.

Set $\X =  \K^* \otimes \F^*$. We look at the vertical and horizontal homology of this bi-complex.

Claim 2. Vertical homology of $\X$ is
   zero for $q < n-g$.

   Notice $(H^q(\X))^p = K_p^* \otimes \Ext_S^{q}(M,S)  = 0$.

Claim 3.  Horizontal homology of $\X$
   zero for $q > n -g$.

    Note $n-g  = \projdim M$. So $F^*_q = 0$.
   Notice
   \[
   (H^p(\X))^q = H^p(\K^*) \otimes F_q^* = 0.
   \]

The result now follows from \ref{ehu}.
\end{proof}
Recall we say that an ideal $I$ is strongly Cohen-Macaulay if all the Koszul-homology modules of $M$ are \CM.
\begin{corollary}
If $I$ is strongly \CM \ ideal of grade $g$ then
$$ \reg H_k(\bx) \leq \sum_{i=1}^{g+k}|x_i|   - g. $$
\end{corollary}
\begin{proof}
We may assume that $K$ is an infinite field.
All  non-zero
 $H_i(\bx)$ have the same dimension as $A = S/I$ (see, e.g., \cite[4.2.2]{Vas}). Let $g = \grade(I)$. Since $I$ is strongly \CM \ we may choose $y_1,\ldots,y_{n-g}$ linear forms in $S$ which are $H_i(\bx)$ regular for $i = 0,\ldots,l-g$. Set $M = S/(\by)$. The result now follows from
Theorem \ref{mt-1} and \ref{mod-reg}.
\end{proof}
An easy consequence of the above result is the following result regarding regularity of zero-dimensional ideals.
\begin{corollary}
If $I$ is a zero dimensional ideal then
$$ \reg H_k(\bx) \leq \sum_{i=1}^{n+k}|x_i|   - n.  $$
\end{corollary}

Next we discuss a curious result.
\begin{corollary}
Let $I$ be a \CM \ ideal. Then
$$\reg S/I \leq \sum_{i=1}^{g}|x_i|   - g. $$
\end{corollary}
\begin{proof}
We may assume that $K$ is an infinite field.
All  non-zero
 $H_i(\bx)$ have the same dimension as $A = S/I$ (see, e.g., \cite[4.2.2]{Vas}). Let $g = \grade(I)$. Since $I$ is  \CM \ we may choose $\by = y_1,\ldots,y_{n-g}$ linear forms in $S$ which are $H_0(\bx) = S/I$ regular and a system of parameters for $H_i(\bx)$ for $i = 1,\ldots,l-g$. Set $M = S/(\by)$. The result now follows from
Theorem \ref{mt-1} and \ref{mod-reg}.
\end{proof}
\section{Proof of Theorem \ref{m2}}
In this section we prove Theorem \ref{m2}. We restate it for the convenience of the reader.
\begin{theorem}
Let $I$ be a zero-dimensional graded ideal in $S$ and let $M$ be a finitely generated graded $S$-module. Then for $k \leq \codim M$,
$$\reg H_k(\bx, M) \leq \sum_{i=1}^{n}|x_i| + t_k(M)  - n    $$
\end{theorem}
\begin{proof}
We consider the second spectral sequence  on $\C$. By \ref{second}  and \ref{com-D} we have
 ${}^{II}E^{pq}_2 = \Ext^p_S(H_q(\bx, M), S)$
Since $\ell( H_i(\bx, M))$ is also finite for all $i \geq 0$ we have
that ${}^{II}E^{pq}_2 = 0$ for all $p \neq n$. So the second spectral sequence on $\C$ collapses. By
\ref{second} it follows that
\begin{equation*}
\Ext^n_S(H_k(\bx, M), S) = H^{n+k}(\D^*) \quad \text{for all}\ k\geq 0. \tag{$\dagger$}
\end{equation*}
By local duality we have
\[
H_k(\bx, M) = \Hom_K\left(\Ext^n_S(H_k(\bx, M), S),  K \right)(n)
\]
So by ($\dagger$) we get that
\[
\reg H_k(\bx, M) = - \mindeg H^{n+k}(\D^*)  -n
\]
Notice $\D^* = \Tot(\K^* \otimes \F^*)$.
We prove that if $k \leq \codim M$ then any homogenous element $\xi$ in $H^{n+k}(\D^*)$ of
\[
\deg \xi < - \sum_{i=1}^{n}|x_i|   - t_{k}(M)
\]
is zero. \\
Let $\xi = [z]$ where $z = (z^{p,q} \mid p+q = n+k)$ is a cycle.

Claim 1: $z^{n,k} = 0$.
   \begin{align*}
\mindeg K^*_{n} \otimes F^*_{k} &= \mindeg K^*_{n} + \mindeg F^*_{k} \\
                    &= \mindeg K^*_{n} \otimes F^*_{k} \\
                    &= - \sum_{i=1}^{n}|x_i|   - t_{k}(M) \\
                    &> \deg \xi
\end{align*}
So $z^{n,k} = 0$.\\
Set $\X = \K^* \otimes \F^*$. We look at the vertical and horizontal homology of this bi-complex.\\
Claim 2: Vertical homology of $\X$ is zero for $q < k \leq \codim M$. \\
Notice $(H^q(\X))^p = K_p^* \otimes \Ext^q_S(M,S) = 0$ for $q < \codim M$.

Claim 3: Horizontal homology of $\X$ is zero for $q > k$. \\
Then note $p < n$
Notice $(H^p(\X))^q = H^p(\bx)\otimes F_q^*  = 0$ for $p < n$.

The result now follows from \ref{ehu}.
\end{proof}
\providecommand{\bysame}{\leavevmode\hbox to3em{\hrulefill}\thinspace}
\providecommand{\MR}{\relax\ifhmode\unskip\space\fi MR }
\providecommand{\MRhref}[2]{%
  \href{http://www.ams.org/mathscinet-getitem?mr=#1}{#2}
}
\providecommand{\href}[2]{#2}

\end{document}